\documentclass{amsart}
\usepackage[mathscr]{euscript}
\usepackage{amsthm}
\usepackage{amssymb}

\begin{document}

{
\def\a{\alpha}
\def\b{\beta}
\def\c{\gamma}
\def\d{\delta}
\def\e{\varepsilon}
\def\A{\mathcal A}
\def\B{\mathcal B}
\def\N{\mathbb N}
\def\R{\mathbb R}

\title{Size-minimal combinatorial designs of staircase type} 
\author{Barbora Bat\'\i kov\'a, Tom\'a\v s J.\ Kepka and Petr C.\ N\v emec}

\address{Barbora Bat\'\i kov\'a, Department of Mathematics, CULS,
Kam\'yck\'a 129, 165 21 Praha 6 - Suchdol, Czech Republic}
\email{batikova@tf.czu.cz}

\address{Tom\'a\v s J.\ Kepka, Faculty of Education, Charles University, M.\ Rettigov\'e 4, 116 39 Praha 1, Czech Republic}
\email{tomas.kepka@karlin.mff.cuni.cz}

\address{Petr C.\ N\v emec, Department of Mathematics, CULS,
Kam\'yck\'a 129, 165 21 Praha 6 - Suchdol, Czech Republic}
\email{nemec@tf.czu.cz}



\begin{abstract}
The minimum size of combinatorial designs of staircase type is found.
\end{abstract}

\maketitle

\newtheorem{lemma}{Lemma}[section]
\newtheorem{rem}[lemma]{Remark}
\newtheorem{theorem}[lemma]{Theorem}
\newtheorem{prop}[lemma]{Proposition}
\newtheorem{obser}[lemma]{Observation}
\newtheorem{ex}[lemma]{Example}

\section{Introduction}

In \cite{B}, an interesting sort of combinatorial designs is constructed which may be called designs of staircase type. Given a positive integer $n$ and a partitioning $n=r_1s_1+\dots+ r_ts_t$, $t,r_i,s_i$ positive integers, such that $r_1>\dots>r_t$ (for $t\ge 2$), we can write $n$ symbols $1,\dots,n$ in the form of a staircase matrix having $r_1$ rows where first $r_1-r_2$ rows have $x_1$ columns, next $r_2-r_3$ rows have $t_1+t_2$ columns, etc., and finally last $r_t$ rows have $t_1+\dots+t_k$ columns. Then we can construct a~design having $r_1+s_1+\dots+s_t$ sets by taking all $r_1$ rows and $s_1+\dots+s_t$ columns of this staircase matrix. Such designs have exactly two replications of each symbol and various cardinalities for the sets constituting the design. Considering designs of this type having the minimum size, the following question jumps out: Given a positive integer $n$, find partitionings $n=r_1s_1+\dots+ r_ts_t$, $t,r_i,s_i$ positive integers, such that $r_1>\dots>r_t$ (for $t\ge 2$) and the sum $r_1+s_1+\dots+s_t$ is the smallest possible. The aim of this note is to find this smallest sum. In fact, the staircase designs are useful in non-adaptive combinatorial group testing (see, e.g., \cite{D}) and this, in turn, has to be accessible to general public use. Not, solely, to a fistful of chosen ones in the know. Henceforth, as a rule, we are making use of ultraelementary methods only.

\section{Preliminaries}

In the whole article, $\mathbb Z$ ($\N_0$, $\N$, $\mathbb Q$, $\mathbb R$, resp.) stand for the domain (halfdomain, halfdomain, field, field, resp.) of integers (non-negative integers, positive integers, rationals, real numbers, resp.) This notation is (with a pinch of salt) standard, while the following one not and it is introduced for local purposes solely.

Let $n\in\N_0$. We denote $\A(n)=\{\,(r,s)\,|\,r,s\in\N_0,r\le s,r+s=n\,\}$ and $\B(n)=\{\,(r,s)\,|\,r,s\in\N_0,r\le s,r\cdot s=n\,\}$ It is immediately visible that $\A(0)=\{(0,0)\}$, $\A(1)=\{(0,1)\}$, $\A(2)=\{(0,2),(1,1)\}$, $\B(0)=\{(0,s)\,|\,s\in\N_0\}$, $\B(1)=\{(1,1)\}$, $\B(2)=\{(1,2)\}$ and $\B(n)\subseteq\N\times\N$ for $n\in\N$. Besides, $\A(0)\cap\B(0)=\{(0,0)\}$, $\A(4)\cap\B(4)=\{(2,2)\}$ and $\A(n)\cap\B(n)=\emptyset$ for $n\ne0,4$.

For $n\ge2$, put $\a(n)=\max\{rs\,|\,(r,s)\in\A(n)\}$ and $\b(n)=\min\{rs\,|\,(r,s)\in\A(n)\}$.

For $n\ge1$, put $\c(n)=\max\,\{\,r+s\,|\,(r,s)\in\B(n)\,\}$ and $\d(n)=\min\,\{\,r+s\,|\,(r,s)\in\B(n)\,\}$.

\begin{prop} {\rm(i)} $\a(n)=\frac{n^2}4\, (=\left(\frac n2\right)^2)$ for every $n\ge 2$, $n$ even.\newline
{\rm(ii)} $\a(n)=\frac{n^2-1}4$ for every $n\ge 3$, $n$ odd.\end{prop}

\begin{proof} We have $\a(2)=1$, $\a(3)=2$, so that we assume $n\ge4$. If $1\le r<2r\le n-1$ then $(r,n-r)\in\A(n)$, $0<(n-2r)^2=n^2-4rn+4r^2$, $4r(n-r)=4rn-4r^2<n^2$, $r(n-r)<(\frac n2)^2$. If $n$ is odd and $2r<n-1$ then $r\le\frac{n-3}2$, $2+2r\le2+n-3=n-1<n$, $2<n-2r$, $4<(n-2r)^2=n^2-4rn+4r^2$, $4r(n-r)=4rn-4r^2<n^2-4<n^2-1$, $r(n-r)<\frac{n^2-1}4$. On the other hand, $\left(\frac n2,\frac n2\right)\in\A(n)$ for $n$ even and $\left(\frac{n-1}2,\frac{n-1}2\right)\in\A(n)$ for $n$ odd. The rest is clear.\end{proof}

\begin{prop} $\b(n)=n-1$ for every $n\ge2$.\end{prop}

\begin{proof} We have $\b(2)=1$, $\b(3)=2$, so that we assume $n\ge4$. Now, $(1,n-1)\in\A(n)$ and, if $(r,s)\in\A(n)$ $2\le r$, then $3\le r+1\le n-1<n$, $r^2-1=(r-1)(r+1)<(r-1)n=rn-n$, $n-1<rn-r^2=r(n-r)=rs$, $n\le rs$.\end{proof}

\begin{prop} $\c(n)=n+1$ for every $n\ge1$.\end{prop}

\begin{proof} Since $(1,n)\in\B(n)$, we have $n+1\le\c(n)$. If $(r,s)\in\B(n)4$, $2\le r$, then $r(r-1)<n(r-1)$, $r^2+n<rn+r$, $r+s=r+\frac nr=\frac{r^2+n}r<\frac{rn+r}r=n+1$.\end{proof}

\begin{rem} \rm If $(r,s)\in\A(n)$, $2\le r$, then $n\le rs$. If $(r,s)\in\B(n)$, $2\le r$, then $r+s\le n$.\end{rem}

\section{First bunch of observations}

\begin{lemma} Let $r_1,r_2\in\R$, $r_1>0$, $r_2\ge0$. Then:\newline
{\rm(i)} $r_1+\frac{r_2}{r_1}\ge2r_2^{\frac12}$.\newline
{\rm(ii)} $r_1+\frac{r_2}{r_1}=2r_2^{\frac12}$ if and only if $r_1=\frac{r_2}{r_1}$ (i.e., $r_2=r_1^2$).\end{lemma}

\begin{proof} $\left(r_1+\frac{r_2}{r_1}\right)^2=r_1^2+2r_2+\frac{r_2^2}{r_1^2}=\left(r_1-\frac{r_2}{r_1}\right)^2+4r_2\ge4r_2$.\end{proof}

\begin{lemma} Let $r_1,r_2\in\N$, $r_1|n$, $r_2|n$.\newline
{\rm(i)} If $r_1+\frac n{r_1}=r_2+\frac n{r_2}$ then either $r_1=r_2$ or $r_1=\frac n{r_2}$, $r_2=\frac n{r_1}$ and $n=r_1r_2$.\newline
{\rm(ii)} If $r_1<r_2$ and $r_2^2\le n$ then $r_2+\frac n{r_2}<r_1+\frac n{r_1}$.\newline
{\rm(iii)} If $r_1^2>n$ then $\left(\frac n{r_1}\right)^2<n$.\end{lemma}

\begin{proof} (i) $r_1^2r_2+nr_2=r_1r_2^2+nr_1$ and $n(r_2-r_1)=r_1r_2(r_2-r_1)$.\newline
(ii) $r_1r_2^2-r_1^2r_2=r_1r_2(r_2-r_1)<r_2^2(r_2-r_1)\le n(r_2-r_1)=nr_2-nr_1$, $r_1r_2\left(r_2+\frac n{r_2}\right)=r_1r_2^2+r_1n<r_1^2r_2+r_2n=r_1r_2\left(r_1+\frac n{r_1}\right)$, $r_2+\frac n{r_2}<r_1+\frac n{r_1}$.\newline
(iii) $\left(\frac n{r_1}\right)^2=\frac{n^2}{r_1^2}<\frac{n^2}n=n$.\end{proof}
 
Being $n\in\N$, let $\e(n)$ designate the greatest integer such that $\e(n)|n$ and $\e(n)^2\le n$. Notice that $\frac n{\e(n)}$ is just the smallest positive integer $m$ dividing $n$ such that $n\le m^2$.

\begin{prop} Let $n\in\N$. Then:\newline
{\rm(i)} $\d(n)=\e(n)+\frac n{\e(n)}\,(=\frac{\e(n)+n}{\e(n)})$.\newline
{\rm(ii)} If $r\in\N$, $r|n$, $r\ne\e(n)$, $r\ne\frac n{\e(n)}$ then $\d(n)<r+\frac nr$.\end{prop}

\begin{proof} Let $(r,s)\in\B(n)$. If $r<\e(n)$ then $\d(n)<r+s$ by 3.2(ii). If $\e(n)<r$ then $n<r^2$, and so $s^2=\frac{n^2}{r^2}<\frac{n^2}n=n$, $s\le\e(n)$. Since $r\le s$, we get $r\le\e(n)$, a contradiction.\end{proof}

\begin{prop} Let $n\in\N$. Then:\newline
{\rm(i)} $\d(n)\ge2n^{\frac12}$.\newline
{\rm(ii)} $\d(n)=2n^{\frac12}$ if and only if $n=k^2$, $k\ge1\ (\d(k^2)=2(k))$.\end{prop}

\begin{proof} We have $n-2\e(n)n^{\frac12}+\e(n)^2=(n^{\frac12}-\e(n))^2\ge0$, so that $n+\e(n)^2\ge2\e(n)n^{\frac12}$ and $\d(n)=\e(n)+\frac n{\e(n)}\ge2n^{\frac12}$ (3.3(i)).\end{proof}

\begin{prop} Let $n\in\N$. Then:\newline
{\rm(i)} $2\le\d(n)\le n+1$.\newline
{\rm(ii)} $\d(n)=n+1$ if and only if either $n=1$ or $n$ is a prime number.\end{prop}

\begin{proof} If $(r,s)\in\B(n)$, $2\le r$, Then $s\le\frac n2$ and $r+s\le r+\frac n2\le s+\frac n2\le n<n+1$.\end{proof}

\begin{ex} \rm The following observation shall
speed up our progress:

\bigskip
$\begin{array}{c |c c c c c c c c c c c c c c c c c c}
n&1&2&3&4&5&6&7&8&9&10&11&12&13&14&15&16&17&18\\
\e(n)&1&1&1&2&1&2&1&2&3&2&1&3&1&2&3&4&1&3\\
\frac n{\e(n)}&1&2&3&2&5&3&7&4&3&5&11&4&13&7&5&4&17&6\\
\d(n)&2&3&4&4&6&5&8&6&6&7&12&7&14&9&8&8&18&9\\
\lfloor\frac n4\rfloor&0&0&0&1&1&1&1&2&2&2&2&3&3&3&3&4&4&4\\
\end{array}$

\bigskip

$\begin{array}{c |c c c c c c c c c c c c c c c c}
n&19&20&21&22&23&24&25&26&27&28&29&30&31&32&33&34\\
\e(n)&1&4&3&2&1&4&5&2&3&4&1&5&1&4&3&2\\
\frac n{\e(n)}&19&5&7&11&23&6&5&13&9&7&29&6&31&8&11&17\\
\d(n)&20&9&10&13&24&10&10&15&12&11&30&11&32&12&14&19\\
\lfloor\frac n4\rfloor&4&5&5&5&5&6&6&6&6&7&7&7&7&8&8&8
\end{array}$

\bigskip
$\begin{array}{c |c c c c c c c c c c c c c c c c}
n&35&36&37&38&39&40&41&42&43&44&45&46&47&48&49&50\\
\e(n&5&6&1&2&3&5&1&6&1&4&5&2&1&6&7&5\\
\frac n{\e(n)}&7&6&37&19&13&8&41&7&43&11&9&23&47&8&7&10\\
\d(n)&12&12&38&21&16&13&42&13&44&15&14&25&48&14&14&15\\
\lfloor\frac n4\rfloor&8&9&9&9&9&10&10&10&10&11&11&11&11&12&12&12
\end{array}$

\bigskip
$\begin{array}{c |c c c c c c c c c c c c c c c c}
n&51&52&53&54&55&56&57&58&59&60&61&62&63&64&65&66\\
\e(n)&3&4&1&6&5&7&3&2&1&6&1&2&7&8&5&6\\
\frac n{\e(n)}&17&13&53&9&11&8&19&29&59&10&61&31&9&8&13&11\\
\d(n)&20&17&54&15&16&15&22&31&60&16&62&33&16&16&18&17\\
\lfloor\frac n4\rfloor&12&13&13&13&13&14&14&14&14&15&15&15&15&16&16&16
\end{array}$

\bigskip
$\begin{array}{c |c c c c c c c c c c c c c c c c}
n&67&68&69&70&71&72&73&74&75&76&77&78&79&80&81&82\\
\e(n)&1&4&3&7&1&8&1&2&5&4&7&6&1&8&9&2\\
\frac n{\e(n)}&67&17&23&10&71&9&73&37&15&19&11&13&79&10&9&41\\
\d(n)&68&21&26&17&72&17&74&39&20&23&18&19&80&18&18&43\\
\lfloor\frac n4\rfloor&16&17&17&17&17&18&18&18&18&19&19&19&19&20&20&20
\end{array}$

\bigskip
$\begin{array}{c |c c c c c c c c c c c c c c c c}
n&83&84&85&86&87&88&89&90&91&92&93&94&95&96&97&98\\
\e(n)&1&7&5&2&3&8&1&9&7&4&3&2&5&8&1&7\\
\frac n{\e(n)}&83&12&17&43&29&11&89&10&13&23&31&47&19&12&97&14\\
\d(n)&84&19&22&45&32&19&90&19&20&24&34&49&24&20&98&21\\
\lfloor\frac n4\rfloor&20&21&21&21&21&22&22&22&22&23&23&23&23&24&24&24
\end{array}$

\bigskip
$\begin{array}{c |c c c c c c c c c c c c c}
n&99&100&101&102&103&104&105&106&107&108&109&110&111\\
\e(n)&9&10&1&6&1&8&7&2&1&9&1&10&3\\
\frac n{\e(n)}&11&10&101&17&103&13&15&53&107&12&109&11&37\\
\d(n)&20&20&102&23&104&21&22&55&108&21&110&21&40\\
\lfloor\frac n4\rfloor&24&25&25&25&25&26&26&26&26&27&27&27&27
\end{array}$

\begin{lemma} Let $1\le n\le111$. Then:\newline
{\rm(i)} $\d(n)<\frac n4$ just for $n=${\rm\ 70, 72, 77, 78, 80, 81, 84, 88, 90, 91, 96, 98, 99, 100, 102, 104, 105, 108, 110}.\newline
{\rm(ii)} $\d(n)=\frac n4$ only for $n=64$.\end{lemma}

\begin{proof} See the foregoing example.\end{proof}

\begin{lemma} Let $n\in\N$ be such that $\e(n)\ge8$. Then:\newline
{\rm(i)} $n\ge64$ and $\d(n)\le\frac n4$.\newline
{\rm(ii)} $\d(n)<\frac n4$, provided that $n\ge65$.\end{lemma}

\begin{proof} Put $r=\e(n)$ and $s=\frac n{\e(n)}$. We have $n=rs$, $8\le r\le s$ and $\d(n)=r+s=\left(\frac{4r}s+4\right)\frac s4\le\frac{8s}4\le\frac{rs}4=\frac n4$. If, moreover, $\d(n)=\frac n4$ then $\frac{4r}s=4$, $r=s$, $n=\e(n)^2$, $8s=rs$, $\e(n)=r=8$ and $n=64$.\end{proof}

\begin{obser} \rm Let $p\in\N$, $p$ a prime.\newline
(i) $\d(p)=p+1>\frac p4$.\newline
(ii) $\d(2p)=p+2>\frac p2=\frac{2p}4$.\newline
(iii) $\d(3p)=p+3>\frac{3p}4$.\newline
(iv) $\d(4p)=p+4>p=\frac{4p}4$.\newline
(v) $\d(5p)=p+5>\frac{5p}4$ for $p\le19$.\newline
(vi) $\d(5p)=p+5<\frac{5p}4$ for $p\ge23$.\newline
(vii) $\d(12)=7>3=\frac{12}4$.\newline
(viii) $\d(6p)=p+6>\frac{6p}4$ for $3\le p\le 11$.\newline
(ix) $\d(6p)=p+6<\frac{6p}4$ for $p\ge13$.\newline
(x) $\d(7p)=p+7>\frac{7p}4$ for $p\le7$.\newline
(xi) $\d(7p)=p+7<\frac{7p}4$ for $p\ge11$.\newline
(xii) $\d(16)=8>4=\frac{16}4$.\newline
(xiii) $\d(24)=10>6=\frac{24}4$.\newline
(xiv) $\d(8p)=p+8>\frac{8p}4$ for $p=5,7$.\newline
(xv) $\d(8p)=p+8<\frac{8p}4$ for $p\ge11$.\end{obser}
\end{ex}

\begin{prop} Let $n\in\N$. Then:\newline
{\rm(i)} $\d(n)<\frac n4$ if and only if the following two conditions are satisfied:\begin{enumerate}
\item[\rm(a)] $n\ne p,2p,3p,4p$ for each prime $p$.
\item[\rm(b)] $n\ne$ {\rm 1, 16, 18, 24, 25, 27, 30, 32, 35, 36, 40, 42, 45, 48, 49, 50, 54, 55, 56, 60, 63, 64, 65, 66, 75, 85, 95} ($27=3^3$ numbers).\end{enumerate}
{\rm(ii)} $\d(n)=\frac n4$ if and only if $n=64$.\end{prop}

\begin{proof} First of all, if $n=p,2p,3p,4p$, $p$ being a prime, then, by 3.9(i),(ii),(iii),(iv), $\d(n)<\frac n4$ . Of course, $\d(1)=2>\frac14$. Assume, therefore, that $n\ne1,2p,3p,4p$. Then, in particular, $n\ge16$. The rest of the proof is divided into four parts:\newline
(1) Let $p|n$, $p\ge11$, $p$ a prime, and let $\d(n)\ge\frac n4$. We show that $n\in\{55,65,66,85,95\}$.

Indeed, it follows from 3.8 that $\e(n)\le7$. Since $p\ge11$, we have $n\ne7p$ by 3.9(xi) and $n<p^2$, $\frac np<p$, $\frac np\le7$, $n\le 7p$. Consequently, $n=5p,6p$. If $n=5p$ then $p=11,13,17,19$ and $n=55,65,85,95$ by 3.9(vi). If $n=6p$ then $=66$ by 3.9(ix).\newline
(2) Let $7|n$ and $\d(n)\ge\frac n4$. We show that $n\in\{35,42,49,56,63\}$.

As $n\ge16$, we have $n=7m$, $m\ge3$. If $m\ge10$ then $7\le\e(n)$, $\d(n)\le7+m<\frac{7m}4$, a contradiction. Thus $5\le m\le9$ and $n=35,42,49,56,63$.\newline
(3) Let $5|n$ and $\d(n)\le\frac n4$. We show that $n=25,30,35,40,45,50,55,60,65,75,85,95$.

We have $n=5m$, $m\ge5$ ($n\ge16$ and $n\ne20=4\cdot5$). If $m\ge21$ then $5\le\e(n)$ and $\d(n)\le5+m<\frac{5m}4=\frac n4$, a contradiction. Tus $5\le m\le20$ and it follows from 3.7 that either $5\le m\le13$ or $m=15,17,19$. Consequently, $n=25,30,35,40,45,50,55,60,65,75,85,95$.\newline
(4) Finally, let $p<5$ whenever $p$ is a prime dividng $n$ and let $\d(n)\ge\frac n4$. We show that $n=16,18,24,27,32,36,48,54,64$ ($9=3^2$ numbers).

It follows from our assumptions that $n=2^a\cdot3^b$, $a,b\in\N_0$, $(a,b)\ne$ (0,0), (0,1), (1,0), (1,1), (0,2), (2,0), (2,1), (3,0). If $\e(n)\ge8$ then $n=64$ due to 3.8 and 3.7. Suppose, henceforth, that $\e(n)\le7$. Then 64 and 81 do not divide $n$, so that $a\le5$, $b\le3$ and $n\le864$. Denote $c=\lfloor\frac a2\rfloor$ and $d=\lfloor\frac b2\rfloor$. Evidently, $c,d\in\N_0$ and $2^c\cdot3^d\le\e(n)\le7$. From this, we conclude easily that $(c,d)=$ (0,0), (0,1), (1,0), (1,1), (2,0) and 
$(a,b)=$ (0,3), (1,2), (1,3), (3,1), (2,2), (2,3), (3,2), (3,3), (4,0), (4,1), (5,0), (5,1). Consequently, $n=$ 27, 18, 54, 24, 36, 108, 72, 216, 16, 48, 32, 96. Since 
$\e(72)=\e(96)=8$, $\e(108)=9$ and $\e(216)=12$, the proof is finished.
\end{proof}

\section{Second bunch of observations}

\begin{prop}
Let $n\in\N$. Then:\newline{\rm(i)} $\d(n)<\frac{n+12}4$ if and only if the following two conditions are satisfied:\begin{enumerate}
\item[\rm(a)] $n\ne p,2p,3p,4p$ for each odd prime $p$.
\item[\rm(b)] $n\ne 4,6,8,16,18,24,25,27,30,35,36,40$.\end{enumerate}
{\rm(ii)} $\d(n)=\frac{n+12}4$ if and only if $n=4,36,40$.
\end{prop}

\begin{proof} We have $\d(4)=4=\frac{4+12}4$, $\d(6)=5=\frac{20}4>\frac{6+12}4$, $\d(8)=6=\frac{24}4>\frac{8+12}4$, $\d(16)=8=\frac{32}4>\frac{16+12}4$, $\d(18)=9=\frac{36}4>\frac{18+12}4$, $\d(24)=10=\frac{40}4>\frac{24+12}4$, $\d(25)=10=\frac{40}4>\frac{25+12}4$, $\d(27)=12=\frac{48}4>\frac{27+12}4$, $\d(30)=11=\frac{44}4>\frac{30+12}4$, $\d(32)=12=\frac{48}4>\frac{32+12}4$, $\d(35)=12=\frac{48}4>\frac{35+12}4$, $\d(36)=12=\frac{48}4>\frac{36+12}4$, $\d(40)=13=\frac{52}4>\frac{4)+12}4$. Furthermore, if $p\ge3$ is a prime then $4\d(p)=4p+4=p+3p+4\ge p+13>p+12$, $4\d(2p)=4p+8=2p+2p+8\ge2p+14>2p+12$, $4\d(3p)=4p+12=3p+p+12\ge3p+15>3p+12$ and $4\d(4p)=4p+16>4p+12$.

Now, conversely, let $n\ne$ 4, 6, 8, 16, 18, 24, 25, 27, 30, 32, 35, 36, 40, $p, 2p, 3p, 4p$, $p\ge3$ being a prime. Our aim is to show that $\d(n)<\frac{n+12}4$. Proceeding by contradiction, assume that $\frac{n+12}4\le\d(n)$. Then $\frac n4<\d(n)$ and, in view of 3.10, we get $n=$ 1, 2, 42, 45, 48, 49, 50, 54, 55, 56, 60, 63, 64, 65, 66, 75, 85, 95. Albeit, $\d(1)=2<\frac{13}4$, $\d(2)=3<\frac{14}4$, $\d(42)=13<\frac{54}4$, $\d(45)=14<\frac{57}4$, $\d(48)=14<\frac{60}4$, $\d(49)=14<\frac{61}4$, $\d(50)=15<\frac{62}4$, $\d(54)=15<\frac{66}4$, $\d(55)=16<\frac{67}4$, $\d(56)=15<\frac{68}4$, $\d(60)=16<\frac{72}4$, $\d(63)=16<\frac{75}4$, $\d(65)=18<\frac{77}4$, $\d(66)=17<\frac{78}4$, $\d(75)=20<\frac{87}4$, $\d(85)=22<\frac{97}4$, $\d(95)=24<\frac{107}4$.\end{proof}

\begin{lemma}
Let $n\in\N$. Then:\newline
{\rm(i)} $\d(n)\le\frac{n+4}2$ if and only if $n$ is not an odd prime number.\newline
{\rm(ii)} $\d(n)=\frac{n+4}2$ if and only if $n=2,4,6,8,10,14,2p$, $p\ge11$ being a prime.\newline
{\rm(iii)} $\d(n)\le\frac{n+3}2$ if and only if $n\ne8,p,2p$, $p$ being a prime.\newline
{\rm(iv)} $\d(n)=\frac{n+3}2$ if and only if $n=1,9$.\newline
{\rm(b)} $\d(n)\le\frac{n+2}2$ if and only if $n\ne1,8,9,p,2p$, $p$ being a prime.\newline
{\rm(vi)} $\d(n)=\frac{n+2}2$ if and only if $n=12$.\newline
{\rm(vii)} $\d(n)\le\frac{n+1}2$ if and only if $n\ne1,8,9,12,p,2p$, $p$ being a prime.\newline
{\rm(viii)} $\d(n)=\frac{n+1}2$ if and only if $n=15$.\newline
{\rm(ix)} $\d(n)\le\frac n2$ if and only if $n\ne1,8,9,12,15,p,2p$, $p$ being a prime.\newline
{\rm(x)} $\d(n)=\frac n2$ if and only if $n=16,18$.\newline
{\rm(xi)} $\d(n)\ne\frac{n-1}2$ if and only if $n\ge20$ and $n\ne p,2p$, $p\ge11$ being a prime.\newline
{\rm(xii)} $\d(n)=\frac{n-1}2$ if and only if $n=21$.\newline
{\rm(xiii)} $\d(n)\le\frac{n-2}2$ if and only if $n\ge20$ and $n\ne21,p,2p$, $p\ge11$ being a prime.\newline
{\rm(xiv)} $\d(n)=\frac{n-2}2$ if and only if $n=20$.\newline
{\rm(xv)} $\d(n)\le\frac{n-3}2$ if and only if $n\ge24$ and $n\ne p,2p$, $p\ge13$ being a prime.\newline
{\rm(xvi)} $\d(n)=\frac{n-3}2$ if and only if $n=27$.\newline
{\rm(xvii)} $\d(n)\le\frac{n-4}2$ if and only if $n\ge24$ and $n\ne27,p,2p$, $p\ge13$ being a prime.\newline
{\rm(xviii)} $\d(n)=\frac{n-4}2$ if and only if $n=24$.
\end{lemma}

\begin{proof} First of all, if $p$ is an odd prime then $\frac{p+4}2<p+1=\d(p)$. Furthermore, $\d(2)=3$, $\frac{2+3}2<3=\frac{2+4}2$, $\d(1)=2$, $\frac{1+2}2<2=\frac{1+3}2$, $\frac{2p+3}2<p+2=\d(2p)=\frac{2p+4}2$. Henceforth, we can assume that $n\ne1,p,2p$, $p$ being any prime. Then, of course, $n\ge8$. If $p\ge11$ then $\d(3p)=p+3<\frac{3p-4}2$, and if $p\ge7$ then $\d(4p)=p+4<\frac{4p-4}2$. Still further, $\frac{9+2}2<6=\d(9)=\frac{9+3}2$, $\frac{15}2<8=\d(15)=\frac{15+1}2$, $\frac{21-2}2<10=\d(21)=\frac{21-1}2$, $\frac{8+3}2<6=\d(8)=\frac{8+4}2$, $\frac{12+1}2<7=\d(12)=\frac{12+2}2$, $\frac{20-3}2<9=\d(20)=\frac{20-2}2$.

Now, we are fully eligible to assume that $n\ne1,p,2p,3p,4p$. Then $n\ge16$ and $\frac n4<\frac{n-4}2$. Taking into account 3.10, it suffices to consider the numbers $n=$ 16, 18, 24, 25, 27, 30, 32, 35, 36, 40, 42, 45, 48, 49, 50, 54, 55, 56, 60, 63, 64, 65, 66, 75, 85, 95. These numbers possess the following values of the function $\d$: 8, 9, 10, 10, 12, 11, 12, 12, 12, 13, 13, 14, 14, 14, 15, 15, 16, 15, 16, 16, 16, 18, 17, 20, 22, 24. The maximal value is 24 henceforth. But, $24\le\frac{n-4}2$ for $n>50$ (from our list). Still remain the numbers, $n=$ 16, 18, 24, 25, 27, 30, 32, 35, 36, 40, 42, 45, 48, 49, 50. For these numbers, the maximal value of $\d$ is 15. But, $15<\frac{n-4}2$ for $n>34$ (from our list again). Actually remaining numbers are $n=$ 16, 18, 24, 25, 27, 30, 32. Here, the maximal value of $\d$ is 12. But, $12<\frac{n-4}2$ for $n=30,32$. Finally, the numbers $n=$ 16, 18, 24, 25, 27 are still in question. We have $\d(27)=12=\frac{27-3}2$, $\d(25)=1)<\frac{25-4}2$, $\d(24)=10=\frac{24-4}2$, $\d)18)=9=\frac{18}2$, $\d(16)=8=\frac{16}2$. Done!\end{proof}

\section{A few consequences}

\begin{prop} Let $n\in\N$. Then $\d(m)<\d(n)$ for every $m\in\N$, $m<n$, if and only if either $n=1$ or $n$ is a prime number.\end{prop}

\begin{proof} The assertion follows immediately from 3.4.\end{proof}

\begin{prop} Let $n\in\N$. Then $\d(m)<\d(n)$ for every composite number $m\in\N$, $m<n$, if and only if either $n=p,2p$, $p$ being a prime, or $n=1,8,21$.\end{prop}

\begin{proof} In view of 5.1, we can assume that the number $n$ is composite as well. If $n=2p$, $p$ being a prime, then $\d(m)\le\frac{m+4}2\le\frac{2p-1+4}2=\frac{2p+3}2=p+\frac32<p+2=\d(2p)$ (use 4.1(i) for $m$). Furthermore, $\d(4)=4<5=\d(6)<6=\d(8)<7=\d(10)=\d(12)<8=\d(15)=\d(16)<9=\d(14)=\d(18)=\d(20)<10=\d(21)$. Consequently, let $n\ne 1,8,21,p,2p$ and let $\d(m)<\d(n)$ for each composite $m$, $m<n$. Having look at the example 3.6, we see readily that $n\ge50$.

Assume, for a moment, that $n=3p,4p$ for a prime $p$. Then $p\ge13$ and we denote by $q$ the smallest prime number greater than $p$. It is generally believed (see 5.3) that $3q<4p$. From that, it follows that $2q<4p-q=3p+(p-q)<3p$, $\d(2q)=q+2\ge p+4>p+3=\d(3p)$ and $\d(3q)=q+3\ge p+5>p+4=\d(4p)$. It means that $n\ne 3p,4p$, a contradiction.

At this stage of our knowledge, we know that $n\ge50$ and $n\ne p,2p,3p,4p$ for every prime $p$. Using 4.1, we get $\d(n)<\frac{n+12}4$. Let $t$ be a prime such that $\frac n2<2t<n$. Then $t+2=\d(2t)<\d(n)<\frac{n+12}4=3+\frac n4<3+t$, so that $t+2<\d(n)<t+3$, where all the three numbers are integers. But this is not possible. Consequently, if $p$ is any prime then either $2p<\frac n2$ or $n\le 2p$. If $n=4k$ ($4k+1,4k+2,4k+3$, resp.) then $k\ge11$ and either $p\le k$ or $2k+1\le p$. This is an apparent contradiction with the illustrious Bertrand Postulate.
\end{proof}

\begin{rem} \rm The following assertion is used in the proof of 5.2: If $p$ is a prime, $p\ge11$, then there is at least one prime $q$ such that $p<q<\frac{4p}3$. (This can be checked easily for $p=11,13,17,19,23,29,31,\dots$ ad libitum.) Concerning this, we can cite \cite{N}. It is shown there that for every $n\ge25$ there is at least one prime $p$ with $n<p<\frac{6n}5$ (of course, $\frac{6n}5<\frac{4n}3$). And there are more results of this type. Unfortunately, all the corresponding proofs are non-elementary.\end{rem}

\begin{prop} Let $n\in\N$. Then $\d(n)<\d(m)$ for every $m\in\N$, $n<m$, if and only if either $n=k^2$ ($\d(n)=2k$) or $n=k^2+k$ ($\d(n)=2k+1$) for some $k\in\N$.\end{prop}

\begin{proof} If $m>k^2$ ($m>k^2+k$) then, by 3.4, $\d(m)\ge2m^{\frac12}>2k=\d(k^2)$ ($\d(m)\ge2m^{\frac12}>2k+1=\d(k^2+k$). Conversely, let $n$ be such that $n\ne l^2$ for every $l\in\N$ and $\d(n)<\d(m)$ for every $m>n$. Let $k$ be the greatest integer satisfying $k^2\le n$. Then $1\le k\le k^2<n<(k+1)^2=k^2+2k+1$, $\d(n)<\d((k+1)^2)=2k+2$, $\d(n)\le 2k+1$. By 3.4, $2k<2n^{\frac12}\le\d(n)$, and so $2k<\d(n)\le 2k+1$, $\d(n)=2k+1$, $2n^{\frac12}\le2k+1$, $n\le4k^2+4k+1$, $n\le k^2+k$. Since $\d(k^2+k)=2k+1$, we conclude that $n=k^2+k$.\end{proof}

\section{Preparatory results}

For every $k\in\N$, define $I_k=\{k^2-k+1,\dots,k^2\}$ and $J_k=\{k^2+1,\dots,k^2+k\}$. We see readily that $|I_k|=k=|J_k|$, the intervals $I_1,I_2,\dots,J_1,J_2,\dots$ are pair-wise disjoint, $\N=\bigcup_{k\in\N} I_k\cup\bigcup_{k\in\N} J_k$ and $I_1<J_1<I_2<J_2<\dots$ (if $A,B\subseteq\N$ then $A<B$ means that $a<b$ for every $a\in A$, $b\in B$).

\begin{lemma} $\min\,\{\,\d(n)\,|\,n\in I_k\,\}=\d(k^2)=2k$.\end{lemma}

\begin{proof} We have $I_1=\{1\}$ and $\d(1)=2$. If $k\ge2$ then $2\le(k-1)^2+(k-1)=k^2-k<n$ and $2k-1=\d((k-1)^2+(k-1))<\d(n)$ by 5.4. Thus $2k\le\d(n)$.\end{proof}

\begin{lemma} Let $k\in\N$. The following conditions are equivalent for $n\in\N$ and $r,s\in\N_0$:\begin{enumerate}
\item[\rm(i)] $n\in I_k$, $(r,s)\in\B(n)$ and $r+s\le2k$.
\item[\rm(ii)] $n\in I_k$, $(r,s)\in\B(n)$ and $r+s=2k$.
\item[\rm(iii)] There is $l\in\N_0$ such that $l^2\le k-1$, $n=k^2-l^2$, $r-k-l$, $s-k+l$.\end{enumerate}\end{lemma}

\begin{proof} (i) implies (ii). We have $2k\le\d(n)\le r+s\le2k$ by 6.1.\newline
(ii) implies (i). Trivial.\newline
(ii) implies (iii) Since $1\le r\le s$ and $r+s=2k$, we have $1\le r\le k$ and $r=k-l$ for some $l$, $0\le l\le k-1$. Then $s=2k-r=k+l$. Moreover, $n\in I_k$ and $rs=n$. Thus $1\le k^2-k+1\le n=(k-l)(k+l)=k^2-l^2\le k^2$ and $0\le l^2\le k-1$.\newline
(iii) implies (ii). Clearly, $0\le l\le k-1$, $1\le r\le s$, $rs-n$ $r+s=2k$ and $(r,s)\in\B(n)$. Finally, $k^2-k+1\le k^2-l^2=n\le k^2$ and $n\in I_k$.\end{proof}

\begin{prop} Let $k\in\N$ and $n\in I_k$. Then:\newline
{\rm(i)} $\d(n)\ge2k$.\newline
{\rm(ii)} $\d(n)=2k$ if and only if $n=k^2-l^2$, $l\in\N_0$, $l^2\le k-1$ $(0\le l\le (k-1)^{\frac12})$.\end{prop}

\begin{proof} Combine 6.1 and 6.2.\end{proof}

\begin{ex} \rm For $k=1$ we get $l=0$ and $n=1$. For $k=2$ we get $l=1,0$ and $n=3,4$. For $k=3$ we get $l=1,0$ and $n=8,9$. For $k=4$ we get $l=1,0$ and $n=15,16$. For $k=5$ we get $l=2,1,0$ and $n=21,24,25$. For $k=6$ we get $l=2,1,0$ and $n=32,35,36$. For $k=7$ we get $l=2,1,0$ and $n=45,48,49$. For $k=8$ we get $l=2,1,0$ and $n=60,63,64$. For $k=9$ we get $l=2,1,0$ and $n=77,80,81$. For $k=10$ we get $l=3,2,1,0$ and $n=91,96,99,100$.\end{ex}

\begin{lemma} $\min\,\{\,\d(n)\,|\,n\in J_k\,\}=\d(k^2+k)=2k+1$.\end{lemma}

\begin{proof} We have $k^2<n\le k^2+k$ and it follows from 5.4 that $2k=\d(k^2)<\d(n)$. Henceforth, $\d(k^2k)=2k+1\le\d(n)$.\end{proof}

\begin{lemma} Let $k\in\N$. The following conditions are equivalent for $n\in\N$ and $r,s\in\N_0$:\begin{enumerate}
\item[\rm(i)] $n\in J_k$, $(r,s)\in\B(n)$ and $r+s\le2k+1$.
\item[\rm(ii)] $n\in J_k$, $(r,s)\in\B(n)$ and $r+s=2k+1$.
\item[\rm(iii)] There is $l\in\N_0$ such that $l^2+l\le k-1$, $n=k^2+k-l^2-l$, $r=k-l$, $s=k+l+1$.\end{enumerate}\end{lemma}

\begin{proof} (i) implies (ii). We have $2k+1\le\d(n)\le r+s\le 2k+1$ by 6.5.\newline
(ii) implies (i). Trivial.\newline
(ii) implies (iii). Since $1\le r\le s$ and $r+s=2k+1$, we have $1\le r\le k$ and $r=k-l$ for some $l$, $0\le l\le k-1$. Then $s=2k+1-r=k+l+1$. Moreover, $n\in J_k$ and $rs=n$. Thus $2\le k^2+1\le n=(k-l)(k+l+1)=k^2-l^2+k-l\le k^2+k$ and $0\le l^2+l\le k-1$.\newline
(iii) implies (ii). Clearly, $0\le l\le k-1$, $1\le r\le s$, $rs=n$, $r+s=2k+1$ and $(r,s)\in\B(n)$. Finally, $k^1+1=(k^2+k)-(k-1)\le k^2+k-(l^2+l)=n\le k^2+k$ and $n\in J_k$.\end{proof}

\begin{prop} Let $k\in\N$ and $n\in J_k$. Then:\newline
{\rm(i)} $\d(n)\le2k+1$.\newline
{\rm(ii)} $\d(n)=2k+1$ if and only if $n=(k^2+k)-(l^2+l)$, $l\in\N_0$, $l^2+l\le k-1$ $\left(0\le l\le\frac{(4k-3)^{\frac12}-1}2\right)$.\end{prop}

\begin{proof} Combine 6.5 and 6.6.\end{proof}

\begin{ex} \rm For $k=1$ we get $l=0$ and $n=2$. For $k=2$ we get $l=0$ and $n=6$. For $k=3$ we get $l=1,0$ and $n=10,12$. For $k=4$ we get $l=1,0$ and $n=18,20$. For $k=5$ we get $l=1,0$ and $n=28,30$. For $k=6$ we get $l=1,0$ and $n=40,42$. For $k=7$ we get $l=2,1,0$ and $n=50,54,56$. For $k=8$ we get $l=2,1,0$ and $n=60,70,72$. For $k=9$ we get $l=2,1,0$ and $n=84,88,90$. For $k=10$ we get $l=2,1,0$ and $n=104,108,110$.\end{ex}

\section{Concluding results} 

\begin{prop} Let $k,n\in\N$ be such that $k^2-k+1\le n\le k^2$ (i.e., $n\in I_k$) and let $t,r_1,\dots,r_t,s_1,\dots,s_t\in\N$ be such that $r_1s_1+\dots+r_ts_t=n$ and $r_1>\dots>r_t$ (for $t\ge2$). Then $r_1+s_1+\dots+s_t\ge2k$.\end{prop}

\begin{proof} Firstly, $r_1(s_1+\dots+s_t)\ge r_1s_1+\dots+s_t=n$ and the equality takes place only for $t=1$. In other words, $w=r_1+s_1+\dots+s_t\ge r+1+\frac n{r_1}$ and $w=r_1+\frac n{r_1}$ just for $t=1$. Now, by 3.1, $r_1+\frac n{r_1}\ge2n^{\frac12}$ and $r_1+\frac n{r_1}=2n^{\frac12}$ just for $n=r_1^2$.

Let $t=1$. Then $w=r_1+s_1=r_1+\frac n{r_1}$ and $\d(n)\le w$. By 6.1, $2k\le\d(n)$. Thus $2k\le w$.

Let $t\ge2$. Then $w>r_1+\frac n{r_1}\le2n^{\frac12}$, $(k-\frac12)^2-k^2-k+\frac14<k^2-k+1\le n\le k^2$, and therefore $k-\frac12<n^{\frac12}\le k$ and $w>2n^{\frac12}>2k-1$. Consequently, $w\ge2k$.\end{proof}

\begin{ex} \rm Let $k\in\N$, $k\ge2$, and let $n\in I_k\setminus\{k^2\}$ (i.e., $k^2-k+1\le n\le k^2-1$). Put $l=n-k^2+k$. Then $1\le l\le k-1$, $n=k^2-k+l$ and $n=k(k-1)+l\cdot1=r_1s_1+r_2s_2$, where $r_1=k$, $r_2=l$, $r_1>r_2$, $s_1=k-1$, $s_2=1$. Of course, $w=r_1+s_1+s_2=k+k-1+1=2k$.\end{ex}

\begin{ex} \rm Let $j,k\in\N$, $j\ge2$, $k\ge\frac{j^2+j+2}2$ ($\ge4$) and let $n\in I_k\setminus\{k^2,k^2-1,\dots,k^2-\frac{j^2+j-2}2\}$ (i.e., $k^2-k+1\le n\le k^2-\frac{j^2+j}2$). Clearly, $3\le\frac{j^2+j}2\le k-1$, $13\le k^2-k+1\le k^2-\frac{j^2+j}2\le k^2-3$ and $2\le \frac{j^2-j+2}2$. Put $l=n-k^2+k+\frac{(j-1)j}2$. Then $n=k^2-k-\frac{(j-1)j}2+l$ and $2\le\frac{j^2-j+2}2=k^2-k+1-k^2+k+\frac{j^2-j}2\le l\le k^2-\frac{j^2+j}2-k^2+k+\frac{(j-1)j}2=k-j\le k-2$. Now, $n=k(k-j)+(k-1)\cdot1+\dots+(k-j+1)\cdot1+l\cdot1=r_1s_1+r_2s_2+\dots+r_js_j+r_{j+1}s_{j+1}$, where $3\le j+1$, $r_1=k$, $r_2=k-1,\dots,r_j=k-j+1$, $r_{j+1}=l$, $r_1>r_2>\dots>r_{j+1}$, $s_1=k-j$, $s_2=\dots=s_{j+1}=1$. Of course, $w=r_1+s_1+\dots+s_{j+1}=k+k-j+j=2k$.\end{ex}

\begin{ex} \rm (i) $32=6^2-6+2=6\cdot5+2\cdot1=6\cdot4+4\cdot2=5\cdot5+4\cdot1+3\cdot1$, $2\cdot6=12=6+5+1=6+4+2=5+5+1+1$.\newline
(ii) $43-7^2-7+1=7\cdot4+6\cdot1+5\cdot1+4\cdot1$, $2\cdot7=14=7+4+1+1+1$.\newline
(iii) $57=8^2-8+1=8\cdot5+7\cdot2+3\cdot1$, $2\cdot8=18=8+5+2+1$.\newline
(iv) $111=11^2-11+1=11\cdot7+10\cdot1+9\cdot1+8\cdot1+7\cdot1$, $2\cdot11=22=11+7+1+1+1+1$.\end{ex}

\begin{rem} \rm(cf. 7.1) Let $k\in\N$ and let $k^2=r_1s_1+\dots r_ts_t$, $t,r_i,s_i\in\N$, $r_1>\dots>r_t$ (for $t\ge2$). Put $w=r_1+s_1+\dots+s_t$. If $t\ge2$ then $w\ge2k+1$ (see the proof of 7.1). If $t=1$ and $w=2k$ then $r_1=k=s_1$.\end{rem}

\begin{ex} \rm (i) $2^2=4=2\cdot1+1\cdot2$, $2>1$, $2+1+2=5=2\cdot2+1$.\newline
(ii) Let $k\ge3$. Then $k^2=(k+2)\cdot1+(k+1)\cdot(k-2)$, $k+2>k+1$, $k_2+1+k-2=2k+1$.\end{ex}

\begin{prop} Let $k,n\in\N$ be such that $k^2+1\le n\le k^2+k$ (i.e., $n\in J_k$) and let $t,r_1,\dots,r_t,s_1,\dots,s_t$ be such that $r_1s_1+\dots+r_ts_t=n$ and $r_1>\dots>r_t$ (for $t\ge2$). Then $r_1+s_1+\dots+s_t\ge2k+1$.\end{prop}

\begin{proof} Using 6.7, we proceed similarly as in the proof of 7.1. If $t\ge2$ then $w>2n^{\frac12}>2k$, and hence $w\ge2k+1$.\end{proof}

\begin{ex} \rm Let $k\in\N$, $k\ge2$, and let $n\in J_k\setminus\{k^2+k\}$ (i.e., $k^2+1\le n\le k^2+k-1$). Put $l=n-k^2$. Then $1\le l\le k-1$, $n=k^2+l$ and $n=k\cdot k+l\cdot1=r_1s_1+r_2s_2$, where $r_1=k$, $r_2=l$, $r_1>r_2$, $s_1=k$, $s_2=1$. Of course, $w=r_1+s_1+s_2=k+k+1=2k+1$.\end{ex}

\begin{ex} \rm Let $j,k\in\N$, $k\ge\frac{j^2+3j+4}2$ ($\ge4$) and let $n\in J_k\setminus\{k^2+k,k^2+k-1\dots,k^2+k-\frac{j^2+3j}2\}$ (i.e., $k^2+1\le n\le k^2+k-\frac{j^2+3j+2}2$). Clearly, $3\le\frac{j^2+3j+2}2\le k-1$, $2\le k^2+1\le k^2+k-\frac{j^2+3j+2}2\le k^2+k-3$ and $2\le\frac{j^2+j+2}2$. Put $l=n-k^2+\frac{j(j+1)}2$. Then $n=k^2-\frac{j(j+1)}2+l$ and $2\le\frac{j^2+j+2}2=k^2+1-k^2+\frac{j(j+1)}2\le l\le k^2+k-\frac{j^2+3j+2}2-k^2+\frac{j(j+1)}2=k-j-1\le k-2$. Now, $n=k\cdot(k-j)+(k-1)\cdot1+\dots+(k-j)\cdot1+l\cdot1=r_1s_1+r_2s_2+\dots+r_{j+1}s_{j+1}+r_{j+2}s_{j+2}$, where $3\le j+2$, $r_1=k$, $r_2=k-1,\dots,r_{j+1}=k-j$, $r_{j+2}=l$, $r_1>r_2>\dots>r_{j+2}$, $s_1=k-j$, $s_2=\dots=s_{j+2}=1$. Of course, $w=r_1+s_1+\dots+s_{j+2}=k+k-j+j+1=2k+1$.\end{ex}

\begin{ex} \rm (i) $37=6^2+1=6\cdot5+4\cdot1+3\cdot1=6\cdot4+5\cdot2+3\cdot1$, $2\cdot6+1=13=6+5+1+1=6+4+2+1$.\newline
(ii) $122=11^2+11=11\cdot8+10\cdot1+9\cdot1+8\cdot1+7\cdot1$, $2\cdot11+1=23=11+8+1+1+1+1$.\end{ex}

\begin{rem} \rm(cf. 7.7) Let $k\in\N$ and let $k^2+k=r_1s_1+\dots+r_ts_t$, $t,r_i,s_i\in\N$, $r_1>\dots>r_t$ (for $t\ge2$). Put $w=r_1+s_1+\dots+s_t$. If $t=1$ and $w=2k+1$ then $(r_1,s_1=(k,k+1),(k+1,k)$ (see 6.6). Now, assume that $t\ge2$. We wish to show that then $w\ge2k+2$.

Assume, on the contrary, that $w=2k+1$. We have $k^2+k<r_1(s_1+\dots+s_t)$, where $s_1+\dots+s_t=2k+1-r_1$. Thus $k^2+k<r_1(2k+1-r_1)=2kr_1+r_1-r_1^2$, $(r_1-k)^2+2kr_1=r_1^2+k^2<2kr_1+r_1-k$, $0\le(r_1-k)^2<r_1-k$, $1\le r_1-k<1$, a contradiction.\end{rem}

\begin{ex} \rm Let $k\ge2$. Then $k^2+k=(k+2)(k-1)+2\cdot1$, $k+2>2$, $k+2+k-1+1=2k+2$.\end{ex}

\section {Summary}

\begin{theorem} Let $n$ be a positive integer and let $\varphi(n)$ denote the minimum of all the sums $r_1+s_1+\dots+s_t$, where $t,r_i,s_i$ are positive integers such that $n=\sum_{i=1}^t r_is_i$ and $r_1>\dots>r_t$ (for $t\ge2$). Then:\newline
{\rm(i)} There is a uniquely determined positive integer $k$ such that $k^2-k+1\le n\le k^2+k$.\newline
{\rm(ii)} If $n\le k^2$ then $\varphi(n)=2k$.\newline
{\rm(iii)} If $k^2+1\le n$ then $\varphi(n)=2k+1$.\end{theorem}

\begin{proof} Combine 7.1, 7.2, 7.7 and 7.8.\end{proof}

\begin{rem} \rm Let $n,k\in\N$. Then $k^2-k+1\le n\le k^2+k$ if and only if $|n-k^2|\le|n-l^2|$ for every $l\in\N$, $l\ne k$.\end{rem}

\begin{ex} \rm Some values of $\varphi(n)$ are presented below:

\bigskip
$\begin{array}{c | c c c c c c c c c c c c c c c c}
n&1&2&3&4&5&6&7&8&9&10&11&12&13&14&15&16\\
\varphi(n)&1&3&4&4&5&5&6&6&6&7&7&7&8&8&8&8
\end{array}$

\bigskip
$\begin{array}{c | c c c c c c c c c}
n&17&18&19&20&100&1000&1000000&1001000&1111111\\
\varphi(n)&9&9&9&9&20&64&2000&2001&2110
\end{array}$
\end{ex}

}
\end{document}